\documentclass{amsart}
\usepackage{amsfonts}

\newtheorem{theorem}{Theorem}
\theoremstyle{plain}

\newtheorem{corollary}{Corollary}

\newtheorem{definition}{Definition}
\newtheorem{example}{Example}

\newtheorem{lemma}{Lemma}

\newtheorem{proposition}{Proposition}

\numberwithin{equation}{section}
\input{tcilatex}

\begin{document}
\title[]{Aspects on weak, $s$-cs and almost injective rings}
\author{Nasr. A. Zeyada}
\address{Cairo University, Faculty of Science, Department of Mathematics,
Egypt}
\curraddr{University of Jeddah, Faculty of Science, Department of
Mathematics, Saudi Arabia}
\email{nasmz@yahoo.com}
\author{Amr K. Amin}
\curraddr{Umm Al-Qura University, University college, Departement of
Mathematics, Saudi Arabia}
\address{Beni-Suef University, Faculty of Science, Departement of
Mathematics, Egypt}
\email{althan72@yahoo.com}
\date{}
\dedicatory{}
\thanks{}

\begin{abstract}
It is not known whether right $CF$-rings ($FGF$-rings) are right artinian
(quasi-Frobenius). This paper gives a positive answer of this question in
the case of weak $CS$ ($s$-$CS$)\ and $GC2$ rings. Also we get some new
results on almost injective rings.
\end{abstract}

\maketitle

\section{Introduction}

A module $M$ is said to satisfy $C1$-condition or called $CS$-module if
every submodule of $M$ is essential in a direct summand of $M$. Patrick F.
Smith \cite{Smith} introduced weak $CS$ modules. A right $R$-module $M$ is
called weak $CS$ if every semisimple submodule of $M$ is essential in a
summand of $M$. I. Amin, M. Yousif and N. Zeyada \cite{Amin} introduced $soc$%
-injective and strongly $soc$-injective modules, for any two modules $M$ and 
$N$, $M$ is $soc$-$N$-injective if any $R$-homomorphism $f:soc(N)%
\longrightarrow M$ extends to $N$. $R$ is called right (self-) $soc$%
-injective, if the right $R$-module $R_{R}$ is $soc$-injective. $M$ is
strongly $soc$-injective if $M$ is $soc$-$N$-injective for any module $N$.
They proved that every strongly $soc$-injective module is weak $CS$.

N. Zeyada \cite{Zeyada} introduced the notion of $s$-$CS$, for any right $R$%
-module $M$, $M$ is called $s$-$CS$ if every semisimple submodule of $M$ is
essential in a summand of $M$.

A ring $R$ is called a right $CF$ ring if every cyclic right $R$-module can
be embedded in a free module. A ring $R$ is called a right $FGF$ ring if
every finitely generated right $R$-module can be embedded in a free right $R$%
-module. In section 2, we show that the right $CF$, weak $CS$ ($s$-$CS$) and 
$GC2$ rings are artinian.

Zeyada, Hussein and Amin introduced the notions of almost and
rad-injectivity \cite{Zey&Huis&Amin}. In the third section we make a
correction to the result \cite[Theorem 2.12]{Zey&Huis&Amin} and we get a new
results using these notions.

Throughout this paper $R$ is an associative ring with identity and all
modules are unitary $R$-modules. For a right $R$-module $M$, we denote the
socle of $M$ by $soc(M)$. $S_{r}$ and $S_{l}$ are used to indicate the right
socle and the left socle of $R$, respectively. For a submodule $N$ of $M$,
the notations $N\subseteq ^{ess}M$ and $N\subseteq ^{\oplus }M$ mean that $N$
is essential and direct summand, respectively. We refer to \cite{And&Fuller}%
, \cite{Ding&Huynh}, \cite{Faith book}, \cite{Kasch} and \cite{Moh&Mull} for
all undefined notions in this paper.

\section{Generalizations of $CS$-modules and rings}

\begin{lemma}
\label{charac of weak}For a right $R$-module $M$, the following statements
are equivalent:

\begin{enumerate}
\item $M$ is weak $CS$.

\item $M=E$ $\oplus T$ where $E$ is $CS$ with $soc(M)\subseteq ^{ess}E$.

\item For every semisimple submodule $A$ of $M$, there is a decomposition $%
M=M_{1}\oplus M_{2}$ such that $A\subseteq M_{1}$ and $M_{2}$ is a
complement of $A$ in $M$.
\end{enumerate}
\end{lemma}

\begin{proof}
$\left( 1\right) \Longrightarrow \left( 2\right) $. Let $M$ be a weak $CS$.
Then $soc\left( M\right) $ is essential in a summand, so $M=E$ $\oplus T$
with $soc(M)\subseteq ^{ess}E$. Now if $K$ is a submodule of $E$, then $%
soc(K)\subseteq ^{ess}L$ where $L$ is a summand of $M$ and $L\subseteq
^{ess}(K+L)$. But $L$ is closed, so $K\subseteq L$. Since $E\subseteq
^{ess}(L+E)$ and $E$ is closed in $M$, so $L\subseteq E$ and $E$ is $CS$.

$\left( 2\right) \Longrightarrow \left( 1\right) $. If $E$ is $CS$ and a
summand of $M$ with $soc(M)\subseteq ^{ess}E$, then every submodule of $%
soc\left( M\right) $ is a summand of $E$ and a summand of $M$.

$(1\Longrightarrow 3)$. Let $A$ be a submodule of $soc\left( M\right) $. By $%
(1)$, there exists $M_{1}\subseteq ^{\oplus }M$ such that $A\subseteq
^{ess}M_{1}$. Write $M=M_{1}\oplus M_{2}$ for some $M_{2}\subseteq M$. Since 
$M_{2}$ is a complement of $M_{1}$ in $M$ and $A$ is essential in $M_{1}$,
then $M_{2}$ is a complement of $A$ in $M$.

$(3\Longrightarrow 1)$. Let $A$ be a submodule of $soc\left( M\right) $. By $%
(2)$, there exists a decomposition $M=M_{1}\oplus M_{2}$ such that $%
A\subseteq M_{1}$ and $M_{2}$ is a complement of $A$ in $M$. Then $(A\oplus
M_{2})\subseteq ^{ess}M=M_{1}\oplus M_{2}$ and $A\subseteq M_{1}$ then $%
A\subseteq ^{ess}M_{1}$. Hence $M$ is weak $CS$ module.
\end{proof}

Recall that, a right $R$-module $M$ is $s$-$CS$ if every singular submodule
of $M$ is essential in a summand \cite{Zeyada}.

\begin{proposition}
\label{charac of s-cs}If $M$ is a right $R$-module, then the following
statements are equivalent:

\begin{enumerate}
\item $M$ is $s$-$CS$.

\item The second singular submodule $Z_{2}\left( M\right) $ is CS and a
summand of M.

\item For every singular submodule $A$ of $M$, there is a decomposition $%
M=M_{1}\oplus M_{2}$ such that $A\subseteq M_{1}$ and $M_{2}$ is a
complement of $A$ in $M$.
\end{enumerate}
\end{proposition}

\begin{proof}
$\left( 1\right) \Longleftrightarrow \left( 2\right) $. \cite[Proposition 14]%
{Zeyada}.

$\left( 1\right) \Longleftrightarrow \left( 3\right) $. Similar argument of
the proof of the above Lemma.
\end{proof}

Given a right $R$-module $M$ we will denote by $\Omega (M)$ [respectively $%
C(M)$] a set of representatives of the isomorphism classes of the simple
quotient modules (respectively simple submodules) of $M$. In particular,
when $M=R_{R}$, then $\Omega (R)$ is a set of representatives of the
isomorphism classes of simple right $R$-modules.\bigskip

\begin{lemma}
\label{P is f. cogen}Let $R$ be a ring, and let $P_{R}$ be a finitely
generated quasi-projective $CS$-module, such that $|\Omega (P)|\leq |C(P)|$.
Then $|\Omega (P)|=|C(P)|$, and $P_{R}$ has finitely generated essential
socle.
\end{lemma}

\begin{proof}
See \cite[Lemma 7.28]{Nich&You}.
\end{proof}

\begin{proposition}
Let $R$ be a ring. Then $R$ is a right $PF$-ring if and only if $R_{R}$ is a
cogenerator and $R$ is weak $CS$.
\end{proposition}

\begin{proof}
Every right $PF$-ring is right self-injective and is a right cogenerator by 
\cite[Theorem 1.56]{Nich&You}.\ Conversely, if $R$ is weak $CS$ and $R$ is
cogenerator then $R=E$ $\oplus T$ where $E$ is $CS$ with $S_{r}\subseteq
^{ess}E$. By the above Lemma, $E$ has a finitely generated, essential right
socle. Since $E$ is right finite dimensional and $R_{R}$ is a cogenerator,
let $S_{r}=S_{1}\oplus S_{2}\oplus .....\oplus S_{m}$ and $I_{i}=I(S_{i})$
be the injective hull of $S_{i}$, then there exists an embedding $\sigma
:I_{i}\longrightarrow R^{I}$ for some set $I$. Then $\pi \circ \sigma \neq 0$
for some projection $\pi :R^{I}\longrightarrow R$, so $(\pi \circ \sigma
)|S_{i}\neq 0$ and hence is monic. Thus $\pi \circ \sigma
:I_{i}\longrightarrow R$ is monic, and so $R$ $=E_{1}\oplus $\thinspace $%
....\oplus E_{m}\oplus T$ where $soc\left( T\right) =0$. So $R$ is a right $%
PF$-ring.
\end{proof}

\begin{proposition}
\cite[Proposition 16]{Zeyada} Let $R$ be a ring. Then $R$ is a right $PF$%
-ring if and only if $R_{R}$ is a cogenerator and $\left( Z_{r}^{2}\right)
_{R}$ is $CS$.
\end{proposition}

\begin{proposition}
The following statements are equivalent:

\begin{enumerate}
\item Every right $R$-module is$\ $weak $CS$.

\item Every right $R$-module with essential socle is $CS$.

\item For every right $R$-module $M$, $M=E\oplus K$ where $E$ is $CS$ with $%
soc\left( M\right) \subseteq ^{ess}E$.
\end{enumerate}
\end{proposition}

\begin{proposition}
The following statements are equivalent:

\begin{enumerate}
\item Every right $R$-module is $\ s$-$CS$.

\item Every \ Goldie torsion right $R$-module is $CS$.

\item For every right $R$-module $M$, $M=Z_{2}\left( M\right) \oplus K$
where $Z_{2}\left( M\right) $ is $CS$.
\end{enumerate}
\end{proposition}

Dinh Van huynh, S. K. Jain and S. R. L\'{o}pez-Permouth \cite{H&J&L.} proved
that if $R$ is simple such that every cyclic singular right $R$-module is $%
CS $, then $R$ is right noetherian.

\begin{corollary}
If \ $R$ is simple such that every cyclic right $R$-module is $s$-$CS$, then 
$R$ is right noetherian.
\end{corollary}

\begin{proposition}
\label{weak+Kasch}If $R$ is a weak $CS$ and $GC2$, and right Kasch, then $R$
is semiperfect.
\end{proposition}

\begin{proof}
Since $R$ is a weak $CS$, so $E$ is $CS$ by Lemma \ref{charac of weak} and $%
R=E\oplus K$\ for some right ideal $K$ of $R$ and so $E$ is a finitely
generated projective module. By Lemma \ref{P is f. cogen}, $E$ has a
finitely generated essential socle. Then, by hypothesis, there exist simple
submodules $S_{1},\cdots ,S_{n}$\ of $E$\ such that $\{S_{1},\cdots ,S_{n}\}$
is a complete set of representatives of the isomorphism classes of simple
right $R$-modules. Since $E$\ is $CS$, there exist submodules $Q_{1},\cdots
,Q_{n}$\ of $E$\ such that $Q_{1},\cdots Q_{n}$\ is an direct summands of $E$%
\ and $(S_{i})_{R}\subseteq ^{ess}(Q_{i})_{R}$\ for $i=1,\cdots ,n$. Since $%
Q_{i}$\ is an indecomposable projective and $GC2$ $R$-module, it has a local
endomorphism ring; and since $Q_{i}$\ is projective, $J(Q_{i})$\ is maximal
and small in $Q_{i}$. Then $Q_{i}$\ is a projective cover of the simple
module $Q_{i}/J(Q_{i})$. Note that $Q_{i}\cong Q_{j}$\ clearly implies $%
Q_{i}/J(Q_{i})\cong Q_{j}/J(Q_{j})$; and the converse also holds because
every module has at most one projective cover up to isomorphism. It is clear
that $Q_{i}\cong Q_{j}$\ if and only if $S_{i}\cong S_{j}$\ if and only if $%
i=j$. Thus, $\{Q_{1}/J(Q_{1}),\cdots ,Q_{n}/J(Q_{n})\}$\ is a complete set
of representatives of the isomorphism classes of simple right $R$-modules.
Hence every simple right $R$-module has a projective cover. Therefore $R$\
is semiperfect.
\end{proof}

The following example show that the proof of \cite[Proposition 13]{Zeyada}
is not true, since the endomorphism ring of an indecomposable projective
module which is an essential extension of a simple module may be not a local
ring. So we add an extra condition that $R$ is right $GC2$ to prove the
Proposition.

\begin{example}
Let $R$ be the ring of triangular matrices, $R=\left\{ \left( 
\begin{array}{cc}
a & 0 \\ 
b & c%
\end{array}%
\right) :a\in Z,b,c\in Q\right\} $. Take $P_{1}=\left\{ \left( 
\begin{array}{cc}
a & 0 \\ 
b & 0%
\end{array}%
\right) :a\in Z,b\in Q\right\} $ and $P_{2}=\left\{ \left( 
\begin{array}{cc}
0 & 0 \\ 
0 & c%
\end{array}%
\right) :c\in Q\right\} $, we see that $P_{1}$ is indecomposable projective
module with simple essential socle\ and $P_{2}$ is projective simple module.
The socle of $P_{1}$ is isomorphic to $P_{2}$ and its endomorphism ring is
isomorphic to $Z$ which is not local.
\end{example}

\begin{proposition}
\label{s+Kasch}If $R$ is right $s$-$CS$ and $GC2$, and right Kasch, then $R$
is semiperfect.
\end{proposition}

\begin{proof}
Since $R$ is a weak $CS$, so $E$ is $CS$ by Lemma \ref{charac of s-cs} and $%
R=E\oplus K$\ for some right ideal $K$ of $R$ and so $E$ is a finitely
generated projective module. By Lemma \ref{P is f. cogen}, $E$ has a
finitely generated essential socle. Then, by hypothesis, there exist simple
submodules $S_{1},\cdots ,S_{n}$\ of $E$\ such that $\{S_{1},\cdots ,S_{n}\}$
is a complete set of representatives of the isomorphism classes of simple
right $R$-modules. Since $E$\ is $CS$, there exist submodules $Q_{1},\cdots
,Q_{n}$\ of $E$\ such that $Q_{1},\cdots Q_{n}$\ is an direct summands of $E$%
\ and $(S_{i})_{R}\subseteq ^{ess}(Q_{i})_{R}$\ for $i=1,\cdots ,n$. Since $%
Q_{i}$\ is an indecomposable projective and $GC2$ $R$-module, it has a local
endomorphism ring; and since $Q_{i}$\ is projective, $J(Q_{i})$\ is maximal
and small in $Q_{i}$. Then $Q_{i}$\ is a projective cover of the simple
module $Q_{i}/J(Q_{i})$. Note that $Q_{i}\cong Q_{j}$\ clearly implies $%
Q_{i}/J(Q_{i})\cong Q_{j}/J(Q_{j})$; and the converse also holds because
every module has at most one projective cover up to isomorphism. It is clear
that $Q_{i}\cong Q_{j}$\ if and only if $S_{i}\cong S_{j}$\ if and only if $%
i=j$. Thus, $\{Q_{1}/J(Q_{1}),\cdots ,Q_{n}/J(Q_{n})\}$\ is a complete set
of representatives of the isomorphism classes of simple right $R$-modules.
Hence every simple right $R$-module has a projective cover. Therefore $R$\
is semiperfect.
\end{proof}

\begin{lemma}
Let $R$ be a semiperfect, left Kasch, left $min$-$CS$ ring. Then the
following hold::

\begin{enumerate}
\item $S_{l}\subseteq _{R}^{ess}R$ and $soc(Re)$ is simple and essential in $%
Re$\ for all local idempotents $e\in R.$

\item $R$ is right Kasch if and only if $S_{l}\subseteq S_{r}$.

\item If $\{e_{1},...,e_{n}\}$ are basic local idempotents in $R$ then $%
\{soc(Re_{1}),...,soc(Re_{n})\}$ is a complete set of distinct
representatives of the simple left $R$-modules.
\end{enumerate}
\end{lemma}

\begin{proof}
See \cite[Lemma 4.5]{Nich&You}.
\end{proof}

Recall that a ring $R$ is right minfull if it is semiperfect, right
mininjective, and $soc(eR)\neq 0$ for each local idempotent $e\in R$.

\begin{corollary}
If $R$ is commutative $s$-$CS$ (weak $CS$) and Kasch, then $R$ is minfull.
\end{corollary}

\begin{proof}
Since every Kasch ring is $C2$, so $R$ is semiperfect by Proposition \ref%
{weak+Kasch} (Proposition \ref{s+Kasch}). Thus using the above Lemma and 
\cite[Proposition 4.3]{Nich&You} $R$ is minfull.
\end{proof}

\begin{theorem}
If $R$\ is right weak $CS$ ($s$-$CS$), $GC2$ and every cyclic right $R$%
-module can be embedded in a free module (right $CF$ ring ) then $R$\ is
right artinian
\end{theorem}

\begin{proof}
If $R$\ is right weak $CS$ ($s$-$CS$) right $CF$, then by Lemma \ref{charac
of weak} (Lemma \ref{charac of s-cs}) $R=E\oplus K$ where $E$ is $CS\ $and $%
soc\left( K\right) =0$ ($Z\left( K\right) =0$). Thus by Proposition \ref%
{weak+Kasch} (Proposition \ref{s+Kasch}), $R$ is semiperfect. The above
Lemma gives $S_{r}\subseteq ^{ess}R_{R}$, so $K=0$. Hence $R$ is $CS$ $\ $%
and $R$ is right artinain by \cite[Corollary 2.9]{Gomez 2}.
\end{proof}

\begin{proposition}
Let $R$ be a right $FGF$, right weak $CS$ ($s$-$CS$) and right $GC2$ ring.
Then $R$ is $QF$.
\end{proposition}

\begin{proof}
It clear by Proposition \ref{weak+Kasch} (Proposition \ref{s+Kasch}) , and 
\cite[Theorem 3.7]{Gomez&Asensio}.
\end{proof}

\section{Almost injective Modules}

\begin{definition}
A right $R$-module $M$ is called almost injective, if $M=E\oplus K$ where $E$
is injective and $K$ has zero radical. A ring $R$ is called right almost
injective, if $R_{R}$ is almost injective.
\end{definition}

In \cite{Zey&Huis&Amin}, the statement of Theorem 2.12 is not true, so the
proof of $\left( 3\right) \Longrightarrow \left( 1\right) $. The following
Proposition is the true version of \cite[Theorem 2.12]{Zey&Huis&Amin}. Then
we rewrite the related results.

\begin{proposition}
\label{semilocal}For a ring $R$ the following are true:

\begin{enumerate}
\item $R$ is semisimple if and only if every almost-injective right $R$%
-module is injective.

\item If $R$ is semilocal, then every rad-injective right R-module is
injective.
\end{enumerate}
\end{proposition}

\begin{proof}
$\left( 1\right) $. Assume that every almost-injective right $R$-module is
injective, then every right $R$-module with zero radical is injective. Thus
every semisimple right $R$-module is injective and $R$ is right $V$-ring.
Hence, every right $R$-module has a zero radical. Therefore, every right $R$%
-module is injective and $R$ is semisimple. The converse is clear.

$\left( 2\right) $. Let $R$ be a semilocal ring and $M$ be a rad-injective
right $R$-module. Consider a homomorphism $f:K\longrightarrow M$ where $K$
is a right ideal of $R$. Since $R$ is semilocal, there exists a right ideal $%
L$ of $R$ such that $K+L=R$ and $K\cap L\subseteq J$ \cite{Lomp}. Then there
exists a $R$-homomorphism $g:R\longrightarrow M$ such that $g\left( x\right)
=f\left( x\right) $ for every $x\in K\cap L$. Define $F:R\longrightarrow M$
by $F\left( x\right) =f\left( k\right) +g\left( l\right) $ for any $x=k+l$
where $k\in K$ and $l\in L$. It is clear that $F$ is a well-defined $R$%
-homomorphism such that $F\mid _{K}=f$. i.e. $F$ extends $f$. Therefore $M$
is injective.
\end{proof}

A ring $R$ is called quasi-Frobenius $(QF)$ if $R$ is right (or left)
artinian and right (or left) self-injective. Also, $R$ is $QF$ if and only
if every injective right $R$-module is projective.

\begin{theorem}
$R$ is a quasi-Frobenius ring if and only if every rad-injective right $R$%
-module is projective.
\end{theorem}

\begin{proof}
If $R$ is quasi-Frobenius, then $R$ is right artinian, and by Proposition %
\ref{semilocal} (2), every rad-injective right $R$-module is injective.
Hence, every rad-injective right $R$-module is projective. Conversely, if
every rad-injective right $R$-module is projective, then every injective
right $R$-module is projective. Thus, $R$ is quasi-Frobenius.
\end{proof}

Recall that a ring $R$ is called a right pseudo-Frobenius ring (right $PF$%
-ring) if the right $R$-module $R_{R}$ is an injective cogenerator.

\begin{proposition}
\label{PF rings}The following are equivalent:

\begin{enumerate}
\item $R$ is a right $PF$-ring.

\item $R$ is a semiperfect right self-injective ring with essential right
socle.

\item $R$ is a right finitely cogenerated right self-injective ring.

\item $R$ is a right Kasch right self-injective ring.
\end{enumerate}
\end{proposition}

\begin{theorem}
\label{Kasch+almost}If $R$ is right Kasch right almost-injective, then $R$
is semiperfect.
\end{theorem}

\begin{proof}
Let $R$ be right Kasch and $R_{R}=E\oplus T$, where $E$ is injective and $T$
has zero radical. If $J=0$, then every simple right ideal of $R$ is
projective and $R$ is semiperfect (for $R$ is right Kasch). Now suppose that 
$J\neq 0$. Clearly, every simple singular right $R$-module embeds in $E$. In
particular, every simple quotient of $E$ is isomorphic to a simple submodule
of $E$, and so $E$ is a finitely generated injective and projective module
containing a copy of every simple quotient of E. By \cite[Lemma 18]%
{Gomez&Asensio}, $E$ has a finitely generated essential socle. Then by
hypothesis, there exist simple submodules $S_{1},...,S_{n}$ of $E$ such that 
$\{S_{1},...,S_{n}\}$ is a complete set of representatives of the
isomorphism classes of simple singular right $R$-modules. Since $E$ is
injective, there exist submodules $Q_{1},...,Q_{n}$ of $E$ such that $%
Q_{1}\oplus $\textperiodcentered \textperiodcentered \textperiodcentered $%
\oplus Q_{n}$ is a direct summand of $E$ and $(S_{i})_{R}\subseteq
^{ess}(Q_{i})_{R}$ for $i=1,2,...,n$. Since $Q_{i}$\ is an indecomposable
injective $R$-module, it has a local endomorphism ring. The projectivity of $%
Q_{i}$ implies that $J(Q_{i})$ is maximal and small in $Q_{i}$. Then $Q_{i}$
is the projective cover of the simple module $Q_{i}/J(Q_{i})$. Note that $%
Q_{i}\cong Q_{j}$ clearly implies $Q_{i}/J(Q_{i})\cong Q_{j}/J(Q_{j})$ and
the converse also holds because every module has at most one projective
cover up to isomorphism. But it is clear that $Q_{i}\cong Q_{j}$ if and only
if $S_{i}\cong S_{j}$, if and only if $i=j$. Moreover, every $Q_{i}/J(Q_{i})$
is singular. Thus, $\{Q_{1}/J(Q_{1}),...,Q_{n}/J(Q_{n})\}$ is a complete set
of representatives of the isomorphism classes of the simple singular right $%
R $-modules. Hence, every simple singular right $R$-module has a projective
cover. Since every non-singular simple right $R$-module is projective, we
conclude that $R$ is semiperfect.
\end{proof}

\begin{proposition}
\label{Rad&Kasch}The following are equivalent:

\begin{enumerate}
\item $R$ is a right $PF$-ring.

\item $R$ is a semiperfect right rad-injective ring with $soc(eR)\neq 0$ for
each local idempotent $e$ of $R$.

\item $R$ is a right finitely cogenerated right rad-injective ring.

\item $R$ is a right Kasch right rad-injective ring.

\item $R$ is a right rad-injective ring and the dual of every simple left $R$%
-module is simple.
\end{enumerate}
\end{proposition}

\begin{proof}
$(1)\Leftrightarrow (2)$ By Proposition \ref{semilocal} (2).

$(1)\Rightarrow (3)$ Clear.

$(3)\Rightarrow (1)$ Since $R$ is a right rad-injective ring, it follows
from \cite[Proposition 2.5]{Zey&Huis&Amin} that $R=E\oplus K$, where $E$ is
injective and $K$ has zero radical. Since $R$ is a right finitely
cogenerated ring, $K$ is a finitely cogenerated right $R$-module with zero
radical. Hence, $K$ is semisimple. Therefore, by \cite[Corollary 8]{You&Zhou}%
, $R$ is a right $PF$-ring.

$(1)\Rightarrow (4)$ Clear.

$(4)\Rightarrow (1)$ If $R$ is right Kasch right rad-injective, then $R$ is
right almost-injective (\cite[Proposition 2.5]{Zey&Huis&Amin}). Thus $R$ is
semiperfect (\ref{Kasch+almost}). Hence $R$ is injective by Proposition \ref%
{semilocal} (2). Therefore, $R$ is right $PF$.

$(1)\Rightarrow (5)$ Since every right $PF$-ring is left Kasch and left
mininjective, the dual of every simple left $R$-module is simple by \cite[%
Proposition 2.2]{mininjective}.

$(5)\Rightarrow (1)$ By \cite[Proposition 2.10]{Zey&Huis&Amin}, $R$ is a
right $min-CS$ ring (i.e., every minimal right ideal of $R$ is essential in
a summand). Thus, by \cite[Theorem 2.1]{Gomez&You}, $R$ is semiperfect with
essential right socle. Proposition \ref{semilocal} (2) entails that $R$ is
right self-injective, and hence right $PF$ by Proposition \ref{PF rings}.

$(1)\Longleftrightarrow (4)$ and $(1)\Longleftrightarrow (5)$ are direct
consequences of \cite[Proposition 2.5]{Zey&Huis&Amin}.
\end{proof}

A result of Osofsky \cite[Proposition 2.2]{Osofisky} asserts that a ring $R$
is $QF$ if and only if $R$ is a left perfect, left and right self-injective
ring. This result remains true for rad-injective rings.

\begin{proposition}
The following are equivalent:

\begin{enumerate}
\item $R$ is a quasi-Frobenius ring.

\item $R$ is a left perfect, left and right rad-injective ring.
\end{enumerate}
\end{proposition}

\begin{proof}
$(1)\Rightarrow (2)$ It is well known.

$(2)\Rightarrow (1)$ By hypothesis, $R$ is a semiperfect right and left
rad-injective ring. By Proposition \ref{semilocal} (2), $R$ is right and
left injective, hence $R$ is quasi-Frobenius.
\end{proof}

Note that the ring of integers $Z$ is an example of a commutative noetherian
almost-injective ring which is not quasi-Frobenius.

\begin{definition}
A ring $R$ is called right $CF$-ring ($FGF$-ring) if every cyclic (finitely
generated) right $R$-module embeds in a free module. It is not known whether
right $CF$-rings ($FGF$-rings) are right artinian (quasi-Frobenius rings).
In the next result, a positive answer is given if we assume in addition that
the ring $R$ is right rad-injective.
\end{definition}

\begin{proposition}
The following are equivalent:

\begin{enumerate}
\item $R$ is quasi-Frobenius.

\item $R$ is right $CF$ and right rad-injective.
\end{enumerate}
\end{proposition}

\begin{proof}
$(1)\Rightarrow (2)$ It is well known.

$(2)\Rightarrow (1)$ Since every simple right $R$-module embeds in $R$, $R$
is a right Kasch ring. By Proposition \ref{Rad&Kasch}, $R$ is right
self-injective with finitely generated essential right socle. Thus, every
cyclic right $R$-module has a finitely generated essential socle, and by 
\cite[Proposition 2.2]{Vamos}, $R$ is right artinian, hence quasi-Frobenius.
\end{proof}

\end{document}